\input ./style/arxiv-vmsta.cfg
\documentclass[numbers,compress,v1.0.1]{vmsta}
\usepackage{vtexbibtags}
\usepackage{authorquery}

\volume{5}
\issue{3}
\pubyear{2018}
\firstpage{297}
\lastpage{316}
\aid{VMSTA111}
\doi{10.15559/18-VMSTA111}
\articletype{research-article}

\numberwithin{equation}{section} 

\startlocaldefs

\allowdisplaybreaks

\newcommand{\N}{\mathbb{N}}
\newcommand{\Z}{\mathbb{Z}}
\newcommand{\R}{\mathbb{R}}
\newcommand{\E}{\mathbb{E}}
\newcommand{\di}{\textup{d}}
\newcommand{\1}{\mathbf{1}}

\newcommand{\schw}{\xrightarrow{\smash{d}}}
\newcommand{\stab}{\xrightarrow{\smash{\mathcal{L}-s}}}
\newcommand{\toop}{\stackrel{\mathbb{P}}{\longrightarrow}}

\endlocaldefs

\newtheorem{theorem}{Theorem}[section]
\newtheorem{proposition}[theorem]{Proposition}
\newtheorem{lemma}[theorem]{Lemma}

\theoremstyle{definition}

\newtheorem{remark}[theorem]{Remark}

\begin{document}

\begin{frontmatter}
\pretitle{Research Article}

\title{A limit theorem for a class of stationary increments L\'evy
moving average process with multiple singularities}

\author{\inits{M.M.}\fnms{Mathias M{\o}rck}~\snm{Ljungdahl}\ead[label=e1]{ljungdahl@math.au.dk}}
\author{\inits{M.}\fnms{Mark}~\snm{Podolskij}\thanksref{cor1}\ead[label=e2]{mpodolskij@math.au.dk}}
\thankstext[type=corresp,id=cor1]{Corresponding author.}
\address{Department of Mathematics, \institution{Aarhus University}, Aarhus, \cny{Denmark}}



\markboth{M.M. Ljungdahl, M. Podolskij}{A limit theorem for a class of stationary increments L\'evy moving average process}

\begin{abstract}
In this paper we present some new limit theorems for power variations of
stationary increment L\'evy driven moving average processes.
Recently, such asymptotic results have been investigated in 
[Ann. Probab. 45(6B) (2017), 4477--4528, Festschrift for Bernt {\O}ksendal, Stochastics 81(1) (2017), 360--383]
under the assumption that the kernel function potentially exhibits a singular
behaviour at $0$. The aim of this work is to demonstrate how some of the
results change when the kernel function has multiple singularity points.
Our paper is also related to the article 
[Stoch. Process. Appl. 125(2) (2014), 653--677]
that studied the same
mathematical question for the class of Brownian semi-stationary models.
\end{abstract}

\received{\sday{2} \smonth{3} \syear{2018}}
\revised{\sday{26} \smonth{6} \syear{2018}}
\accepted{\sday{29} \smonth{7} \syear{2018}}
\publishedonline{\sday{20} \smonth{8} \syear{2018}}
\end{frontmatter}

\section{Introduction}
\label{sec1}

In recent years limit theorems and statistical inference for high
frequency observations of stochastic processes have received a great
deal of attention. The most prominent class of high frequency
statistics are power variations that have been proved to be of immense
importance for the analysis of the fine structure of an underlying
stochastic process. The asymptotic theory for power variations and
related statistics has been intensively studied in the setting of It\^o
semimartingales, fractional Brownian motion and Brownian
semi-stationary processes, to name just a few; see for example \cite
{BCP11,BCP13,BGJPS05,C01,JP12} among many others.

In the recent work \cite{BLP17,BP17} power variations of stationary
increments L\'evy moving average processes have been investigated in
details. These are continuous-time stochastic processes $(X_t)_{t \geq
0}$, defined on a probability space $(\varOmega, \mathcal{F}, \mathbb
{P})$, that are given by
\begin{equation}
\label{Xmodel} X_t = \int_{-\infty}^{t}
\bigl(g(t - s) - g_0(-s) \bigr) \, \di L_s,
\end{equation}
where $L = (L_t)_{t \in\R}$ is a symmetric L\'evy process on $\R$ with
$L_0 = 0$ and without Gaussian component. Moreover, $g, g_0 : \R\to\R
$ are deterministic functions vanishing on $(-\infty, 0)$. The most
prominent subclasses include L\'evy moving average processes, which
correspond to the setting $g_0 = 0$, and the linear fractional stable
motion, which is obtained by taking $g(s) = g_0(s) = s^\alpha_+$ and
$L$ being a symmetric $\beta$-stable L\'evy process with $\beta\in(0,
2)$. The latter is a self-similar process with index $H=\alpha+1/\beta
$; see \cite{ST94}.

We introduce the $k$th order increments $\Delta_{i, k}^{n} X$ of $X$,
$k \in\N$, that are defined by
\begin{equation}
\label{filter} \Delta_{i, k}^{n} X := \sum
_{j = 0}^k (-1)^j \binom{k}{j}
X_{(i - j) /
n}, \quad i \geq k.
\end{equation}
For example, we have that $\Delta_{i, 1}^n X = \smash{X_{\frac{i}{n}}}
- \smash{X_{\frac{i - 1}{n}}}$ and $\Delta_{i, 2}^n X = \smash{X_{\frac
{i}{n}}} - 2 \smash{X_{\frac{i - 1}{n}}} + \smash{X_{\frac{i -
2}{n}}}$. The main statistic of interest is the power variation
computed on the basis of $k$th order increments:
\begin{equation}
\label{vn} V(X, p; k)_n := \sum_{i = k}^n
|\Delta_{i, k}^{n} X|^p, \quad p > 0.
\end{equation}
A variety of asymptotic results has been shown for the statistic
$V(p;k)_n$ in \cite{BLP17,BP17}. The mode of convergence and possible
limits heavily depend on the interplay between the power $p$, the form
of the kernel function $g$ and the Blumenthal--Getoor index of $L$. We
recall that the Blumenthal--Getoor index is defined via
\begin{equation}
\label{BG} \beta:= \inf \Biggl\{r \geq0 : \int_{-1}^1
|x|^r \, \nu(\di x) < \infty \Biggr\} \in[0, 2],
\end{equation}
where $\nu$ denotes the L\'evy measure of $L$. It is well known that
$\sum_{s \in[0,1]} |\Delta L_s|^p$ is finite when $p > \beta$, while
it is infinite for $p < \beta$. Here $\Delta L_s = L_s - L_{s-}$ where
$L_{s-} = \lim_{u \uparrow s, u < s} L_u$. To formulate the results of
\cite{BLP17,BP17}, we introduce the following set of assumptions on
$g$, $g_0$ and $\nu$:
\medbreak
\noindent\textbf{Assumption~(A):} The function $g : \R\to\R$
satisfies the condition
\begin{equation}
\label{kshs} g(t) \sim c_0 t^\alpha\quad\text{as } t
\downarrow0 \quad\text{for some } \alpha> 0\text{ and } c_0 \neq0,
\end{equation}
where $g(t) \sim f(t)$ as $t\downarrow0$ means that $\lim_{t
\downarrow0} g(t) / f(t) = 1$. For some $w\in(0, 2]$, $\limsup_{t \to
\infty} \nu(x : |x| \geq t) t^{w} < \infty$ and $g - g_0$ is a bounded
function in $L^{w}(\R_+)$. Furthermore, $g$ is $k$-times continuous
differentiable on $(0, \infty)$ and there exists a $\delta> 0$ such
that $|g^{(k)}(t)| \leq K t^{\alpha- k}$ for all $t \in(0, \delta)$,
$|g^{(k)}|$ is decreasing on $(\delta, \infty)$ and $g^{(j)} \in
L^w((\delta, \infty))$ for $j \in\{1, k\}$.
\medbreak
\noindent\textbf{Assumption~(A-log):} In addition to (A) suppose that
\[
\int_\delta^\infty|g^{(k)}(s)|^w
\big| \log\bigl(|g^{(k)}(s)|\bigr)\big| \, \di s < \infty.
\]
Intuitively speaking, Assumption~(A) says that $g^{(k)}$ may
have a singularity at $0$ when $\alpha$ is small, but it is smooth
outside of $0$. The theorem below has been proved in \cite{BLP17,BP17}. We recall that a sequence of $\R^d$-valued random variables
$(Y_n)_{n\geq1}$ is said to converge stably in law to a random
variable $Y$, defined on an extension of the original probability space
$(\varOmega, \mathcal{F}, \mathbb{P})$, whenever the joint convergence in
distribution $(Y_n, Z) \schw(Y,Z)$ holds for any $\mathcal
F$-measurable $Z$; in this case we use the notation $Y_n \stab Y$. We
refer to \cite{AE78,R63} for a detailed exposition of stable convergence.

\begin{theorem}[{\cite[Theorem~1.1(i)]{BLP17} and \cite[Theorem~1.2(i)]{BP17}}]
\label{th1}
Suppose that Assumption~(A) holds, the Blumenthal--Getoor index
satisfies $\beta< 2$ and $p > \beta$. If $w = 1$ assume that (A-log)
holds. Then we obtain the following cases:
\begin{enumerate}
\item\label{it:th1:1} When $\alpha<k-1/p$ then we have the stable convergence
\begin{equation}
\label{part1} %
\begin{aligned} n^{\alpha p} V(X, p;
k)_n &\stab|c_0|^p \sum
_{m : T_m \in[0, 1]} |\Delta L_{T_m}|^p
V_m
\\
\quad\text{with} \quad V_m &= \sum_{l = 0}^{\infty}
|h_k(l + U_m)|^p, \end{aligned} %
\end{equation}
where $(T_m)_{m\geq1}$ denote the jump times of $L$, $(U_m)_{m \geq
1}$ is a sequence of independent identically (i.i.)
$\mathcal{U}(0,1)$-distributed\querymark{Q2} variables independent of
$L$, and the function $h_k$ is defined by
\begin{equation}
\label{def-h} h_k(x) = \sum_{j = 0}^k
(-1)^j \binom{k}{j} (x - j)_{+}^{\alpha} \quad
\text{with} \quad y_+ = \max\{y, 0\}.
\end{equation}

\item\label{it:th1:2} When $\alpha=k-1/p$ and additionally
$1/p+1/w>1$, then we have
\begin{equation}
\label{part2} %
\begin{aligned} \frac{n^{\alpha p}}{\log(n)} V(X,p;k)_n
&\toop | c_0 q_{k, \alpha} |^p \sum
_{s \in(0, 1]} |\Delta L_s|^p
\\
\quad\text{with} \quad q_{k, \alpha} &:= \prod_{j = 0}^{k - 1}
(\alpha- j). \end{aligned} %
\end{equation}
\end{enumerate}
\end{theorem}
We remark that the first order asymptotic theory of \cite
[Theorem~1.1]{BLP17} includes two more regimes: an ergodic type limit
theorem in the setting $p < \beta$, $\alpha< k - 1 / \beta$ and
convergence in probability to a random integral in the setting $p
\geq1$, $\alpha> k - 1 / \max\{p, \beta\}$. However, in this paper we
concentrate ourselves on results of Theorem~\ref{th1}, which are quite
non-standard in the literature. More specifically, our aim is to extend
the theory of Theorem~\ref{th1} to kernels $g$ that exhibit multiple
singularities. We call a point $x \in\R_+$ a singularity point when
the $k$th derivative $g^{(k)}$ of $g$ explodes at $x$. Note that under
Assumption~(A) and condition $\alpha\leq k - 1 / p$ the function $g$
has only one singularity point at $x=0$. In practical applications a
singularity point $x \in\R_+$ leads to a strong feedback effect
stemming from the past jumps around the time $t-x$. Such effects has
been discussed in the context of turbulence modelling in \cite{GP14}.

We will show that the limits in Theorem~\ref{th1}\ref{it:th1:1} and~\ref
{it:th1:2} will be affected by the presence of multiple singularity
points. More precisely, we will see that the increments $\Delta
_{i,k}^{n} X$ can be heavily influenced by the jumps of $L$ that
happened in the past, and the time delay is determined by the
singularity points of $g$. The obtained result is similar in spirit to
the work \cite{GP14} that studied quadratic variation of Brownian
semi-stationary processes under multiple singularities of the kernel
$g$. Furthermore, we will prove that in general the stable convergence
in Theorem~\ref{th1}\ref{it:th1:1} only holds along a subsequence.

The paper is structured as follows. Section~\ref{sec2} presents the
main results of the article. Proofs are collected in Section~\ref{sec3}.

\section{Main results}
\label{sec2}

We consider stationary increments L\'evy moving average processes as
defined at \eqref{Xmodel} and recall that the driving motion $L$ is a
pure jump L\'evy process with L\'evy measure $\nu$. Now, we introduce
the condition on the kernel function $g$:
\medbreak
\noindent\textbf{Assumption~(B):} For some $w \in(0, 2]$, $\limsup_{t
\to\infty} \nu(x : |x| \geq t) t^w < \infty$ and $g - g_0$ is a
bounded function in $L^w(\R_+)$. Furthermore, there exist points $0 =
\theta_0 < \theta_1 < \cdots< \theta_l$ such that the following
properties hold:
\begin{enumerate}
\item\label{it:B:1} $g(t) \sim c_0 t^{\alpha_0}$ as $t \downarrow0$
for some $\alpha_0 > 0$ and $c_0 \neq0$.

\item\label{it:B:2} $g(t) \sim c_z |t - \theta_z|^{\alpha_z}$ as $t \to
\theta_z$ for some $\alpha_z > 0$ and $c_z \neq0$, and for all $z = 1,
\ldots, l$.

\item\label{it:B:3} $g \in C^k(\R_+ \setminus\{\theta_0, \ldots, \theta
_l\})$.

\item\label{it:B:4} There exist $\delta$, $K > 0$ such that
$|g^{(k)}(t)| \leq K |t - \theta_z|^{\alpha_z - k}$ for all $t \in
(\theta_z - \delta, \theta_z + \delta) \setminus\{\theta_z\}$, for any
$z = 0, \ldots, l$. Furthermore, there exists a $\delta' > 0$ such that
$|g^{(k)}|$ is decreasing on $(\theta_l + \delta', \infty)$ and
$g^{(j)} \in L^w ((\theta_l + \delta', \infty))$ for $j \in\{1, k\}$.
\end{enumerate}
Let us give some remarks on Assumption~(B). First of all,
conditions~(B)\ref{it:B:1} and~(B)\ref{it:B:2}, which are direct
extensions of \eqref{kshs}, mean that for small powers $\alpha_z > 0$
the points $\theta_z$ are singularities of $g$ in the sense that
$g^{(k)}(\theta_z)$ does not exist. On the other hand, condition~(B)\ref
{it:B:3} states that there exist no further singularities. The
parameter $w$ is by no means unique. It simultaneously describes the
tail behaviours of the L\'evy measure $\nu$ and the integrability of
the function $|g^{(k)}|$, which exhibit a trade-off. When $L$ is $\beta
$-stable we always take $w=\beta$. Furthermore, Assumption~(B)
guarantees the existence of $X_t$ for all $t \geq0$. Indeed, it
follows from \cite[Theorem~7]{RR89} that the process $X$ is
well-defined if and only if for all $t \geq0$,
\begin{equation}
\label{sdkfhsd;kf} \int_{-t}^\infty\int
_{\R} \bigl(| f_t(s)x |^2 \wedge1
\bigr) \, \nu (\di x) \, \di s < \infty,
\end{equation}
where $f_t(s) = g(t + s) - g_0(s)$. By adding and subtracting $g$ to
$f_t$ it follows by Assumption~(B) and the mean value theorem that
$f_t$ is a bounded function in $L^w(\R_+)$. For all $\epsilon> 0$,
Assumption~(B) implies that
\[
\int_{\R} \bigl(|y x|^2 \wedge1\bigr) \,\nu(
\di x) \leq K \bigl( \1_{\{|y| \leq
1\}} |y|^w + \1_{\{|y| > 1\}}|y|^{\beta+ \epsilon}
\bigr),
\]
which shows \eqref{sdkfhsd;kf} since $f_t$ is a bounded function in
$L^w(\R_+)$.

\begin{remark}[Toy example]
\label{rem1}
Recall the following well-known results about the power variation of a
pure jump L\'evy process $L$:
\[
V(L,p;k)_n \toop\sum_{s \in[0,1]} |\Delta
L_s|^p < \infty
\]
for any $k \geq1$ and any $p > \beta$. Let us now consider a simple
stationary increments L\'evy moving average process $X$ with $g_0 = 0$
and $g(x) =
\1_{[0, 1]}(x)$. In this case we may call the points $\theta_0 = 0$ and
$\theta_1 = 1$ the singularities of $g$, although they do not precisely
correspond to conditions~(B)\ref{it:B:1} and~(B)\ref{it:B:2}, and we
observe that $X_t = L_t - L_{t - 1}$. Hence, we obtain the convergence
in probability
\[
V(X,p;k)_n \toop\sum_{s \in[0, 1]} |\Delta
L_s|^p + \sum_{s \in[-1,
0]} |\Delta
L_s|^p
\]
for any $k \geq1$ and any $p > \beta$. This result demonstrates that
even in the simplest setting multiple singularities lead to a different limit.
\end{remark}
It turns out that only the minimal powers among $\{\alpha_0,
\ldots, \alpha_l\}$ determine the asymptotic behaviour of the statistic
$V(X, p; k)_n$. Thus, we define
\begin{equation}
\label{Aalpha} \alpha:= \min\{\alpha_0, \ldots,
\alpha_l\} \qquad\text{and} \qquad \mathcal{A} := \{z :
\alpha_z = \alpha\}.
\end{equation}
Furthermore, we introduce the notation $h_{k, 0} := h_k$ and
\begin{equation}
\label{hkz} h_{k, z}(x) = \sum_{j = 0}^k
(-1)^j \binom{k}{j} |x - j|^{\alpha_z} \quad\text{for } z =
1, \ldots, l.
\end{equation}
In the main result below we consider a subsequence $(n_j)_{j \in\N}$
such that the following condition holds:
\begin{equation}
\label{nj} \lim_{j \to\infty} \{n_j
\theta_z\} = \eta_z \in[0, 1] \quad\text {for all } z
\in\mathcal{A},
\end{equation}
where $\{x\}$ denotes the fractional part of $x \in\R$. Obviously,
such a subsequence always exists since $\{n \theta_z\}$ is a bounded
sequence. Sometimes we will require a stronger condition, which is
analogous to Assumption~(A-log):
\medbreak
\noindent\textbf{Assumption~(B-log):} Condition~(B) holds and we have that
\[
\int_{\theta_l+\delta'}^{\infty} |g^{(k)}
(t)|^{w} \big|\log\bigl(|g^{(k)} (t)|\bigr)\big| \, \di t < \infty.
\]
The main result of the paper is the following theorem.

\begin{theorem}
\label{th2}
Suppose that Assumption~(B) holds, the Blumenthal--Getoor index
satisfies $\beta< 2$ and $p > \beta$. If $w = 1$ assume that (B-log)
holds. Recall the notations \eqref{Aalpha} and~\eqref{hkz}. Then we
obtain the following cases:
\begin{enumerate}
\item\label{it:th2:1} When $\max_{0\leq z \leq l} \alpha_z <k-1/p$ and
condition~\eqref{nj} holds, then we have the stable convergence
\begin{align}
\label{part1.2} &n_j^{\alpha p}V(X,p;k)_{n_j} \stab\sum
_{z \in\mathcal{A}} |c_z|^p \sum
_{m : T_m \in[-\theta_z, 1 - \theta_z]} |\Delta L_{T_m}|^p
V_m^z
\\
&\text{with} \quad V_m^z = \sum
_{r\in\Z} \big|h_{k,z}\bigl(r+1-\{U_m+
\eta_z\} \bigr)\big|^p.
\nonumber
\end{align}
as $j \to\infty$, where $(U_m)_{m \geq1}$ is a sequence of i.i. $\mathcal
{U}(0,1)$-distributed variables independent of $L$.

\item\label{it:th2:2} Let $\alpha= \alpha_0 = \cdots= \alpha_l = k -
1/p$. Assume that the functions $f_z : \R_+ \to\R$ defined by $f_z(x)
= g(x) / |x - \theta_z|^{\alpha}$ are in $ C^k((\theta_z - \delta,
\theta_ z + \delta))$ for all $\delta< \max_{1 \leq j \leq l}(\theta_j
- \theta_{j - 1})$. If $1/p+1/w>1$, then we have
\begin{equation}
\label{part2.2} \frac{n^{\alpha p}}{\log(n)} V(X, p; k)_n \toop|q_{k, \alpha}
|^p \sum_{z = 0}^{l}
|c_z|^p (1 + \1_{\{z \geq1\}} ) \sum
_{m : T_m \in[-\theta
_z, 1 - \theta_z]} |\Delta L_{T_m}|^p,
\end{equation}
where the constant $q_{k,\alpha}$ has been introduced in Theorem~\ref
{th1}\ref{it:th1:2}.
\end{enumerate}
\end{theorem}
We remark that the stable convergence in Theorem~\ref{th2}\ref
{it:th2:1} only holds along the subsequence $(n_j)_{j \geq1}$, which
is seen from the form of the limit in \eqref{part1.2} that depends on
$(\eta_z)$. The original statistic $n^{\alpha p} V(X, p; k)_{n}$ is
tight, but does not converge except when $\theta_z \in\N$ for all $z
\in\mathcal{A}$. On the other hand, in Theorem~\ref{th2}\ref{it:th2:2}
we do not require to consider a subsequence.

Notice that the interval $ [-\theta_z, 1 - \theta_z]$, which appears in
Theorem~\ref{th2}, is the set $[0, 1]$ shifted by $\theta_z$ to the
left. Given the discussion of Remark~\ref{rem1}, such a shift in the
limit is not really surprising. We recall that a similar phenomenon has
been discovered in \cite{GP14} in the context of Brownian
semi-stationary processes. These are stochastic processes $(Y_t)_{t
\geq0}$ defined by
\[
Y_t = \int_{-\infty}^{t} g(t - s)
\sigma_s \, \di W_s,
\]
where $W$ is a two-sided Brownian motion and $(\sigma_t)_{t \in\R}$ is
a c\'adl\'ag process. When the kernel function $g$ satisfies
conditions~(B)\ref{it:B:1} and~(B)\ref{it:B:2} along with some further
assumptions, which in particular ensure the existence of $Y_t$, the
authors have shown the following convergence in probability (see \cite
[Theorem 3.2]{GP14}):
\begin{equation}
\frac{1}{n \tau_n^2} V(Y, 2; k)_n \toop\sum
_{z \in\mathcal{A}} \pi_z \int_{-\theta_z}^{1 - \theta_z}
\sigma_s^2 \, \di s
\end{equation}
where $\tau_n^2 = \E[(\Delta_{k, k}^n G)^2]$ with $G_t = \int_{-\infty
}^{t} g(t - s) \, \di W_s$, and the probability weights $(\pi_z)_{z \in
\mathcal{A}}$ are given by
\[
\pi_z = \frac{c_z^2 \| h_{k, z}\|^2_{L^2(\R)}}{\sum_{z \in\mathcal{A}}
c_z^2 \| h_{k, z} \|^2_{L^2(\R)}}.
\]
Hence, we observe the same shift phenomenon in the integration region
as in Theorem~\ref{th2}.

\section{Proofs}
\label{sec3}

Throughout this section all positive constants are denoted by $C$
although they may change from line to line. We will divide the proof of
Theorem~\ref{th2} into several steps. First, we will show the
statements \eqref{part1.2} and~\eqref{part2.2} for a compound Poisson
process. In the second step we will decompose the jump measure of $L$
into jumps that are bigger than $\epsilon$ and jumps that are smaller
than $\epsilon$. The big jumps form a compound Poisson process and
hence the claim follows from the first step. Finally, we prove
negligibility of small jumps when $\epsilon\to0$.

We start with an important proposition.

\begin{proposition}
\label{prop1}
Let $T = (T_1, \ldots, T_d)$ be a stochastic vector with a density $v :
\R^d \to\R_+$. Suppose there exists an open convex set $A \subseteq\R
^d$ such that $v$ is continuously differentiable on $A$ and vanishes
outside of $A$. Then, under condition~\eqref{nj}, it holds that
\begin{equation}
\label{mainprop} \bigl(\{n_j T + n_j \theta_z
\}\bigr)_{z \in\mathcal{A}} \stab\bigl(\{U + \eta_z\}
\bigr)_{z \in\mathcal{A}} \quad\text{as } j \to\infty,
\end{equation}
where $\{x\}$ denotes the fractional parts of the vector $x \in\R^d$
and $x + a$, $a \in\R$, is componentwise addition. Here $U = (U_1,
\ldots, U_d)$ consists of i.i. $\mathcal{U}(0, 1)$-distributed random
variables defined on an extension of the space $(\varOmega, \mathcal{F},
\mathbb{P})$ and being independent of $\mathcal{F}$.
\end{proposition}

\begin{proof}
We first show the stable convergence
\begin{equation}
\label{stabconv} \{ nT \} \stab U.
\end{equation}
This statement has already been shown in \cite[Lemma 4.1]{BLP17}, but
we demonstrate its proof for completeness. Let $f:\R^d \times\R^d\to\R
$ be a $C^1$-function, which vanishes outside some closed ball in
$A\times\R^d$. We claim that there exists a finite constant $K > 0$
such that for all $\rho> 0$
\begin{equation}
\label{eq-est-f-g} D_\rho:= \Bigg| \int_{\R^d} f\bigl(x, \{
x/\rho\}\bigr) v(x) \, \di x - \int_{\R
^k} \biggl(\int
_{[0, 1]^d} f(x,u) \, \di u \biggr) v(x) \, \di x \Bigg| \leq K \rho.
\end{equation}
By \eqref{eq-est-f-g} used for $\rho= 1/n$ we obtain that
\begin{equation}
\label{sdkfh} \E\bigl[f\bigl(T, \{n T\}\bigr)\bigr] \longrightarrow\E\bigl[f(T,
U)\bigr]\quad\text{as } n \to \infty.
\end{equation}
Moreover, due to \cite[Proposition~2(D'')]{AE78}, \eqref{sdkfh} implies
the stable convergence\break $\{n T\} \stab U$ as $n\to\infty$. Thus, we need
to prove the inequality \eqref{eq-est-f-g}. Define $\phi(x,u) := f(x,
u) v(x)$. Then it holds by substitution that
\begin{align*}
\int_{\R^d} f\bigl(x, \{ x/\rho\}\bigr) v(x) \, \di x &= \sum
_{j \in\Z^d} \int_{(0, 1]^d}
\rho^d \phi(\rho j + \rho u, u) \, \di u
\end{align*}
and
\begin{align*}
\int_{\R^d} \biggl(\int_{[0, 1]^d}
f(x,u) \, \di u \biggr) v(x) \, \di x &= \sum_{j \in\Z^d}
\int_{[0, 1]^d} \biggl(\int_{(\rho j, \rho(j +
1)]} \phi(x, u) \,
\di x \biggr) \, \di u.
\end{align*}
Hence, we conclude that
\begin{align*}
D_\rho&\leq\sum_{j \in\Z^d} \int
_{(0, 1]^d} \Bigg|\int_{(\rho j,
\rho(j + 1)]} \phi(x, u) \, \di x -
\rho^d \phi(\rho j + \rho u , u) \Bigg| \, \di u
\\
&\leq\sum_{j \in\Z^d} \int_{(0, 1]^d} \int
_{(\rho j,\rho(j + 1)]} \big|\phi(x, u)- \phi(\rho j + \rho u, u) \big| \, \di x \, \di
u.
\end{align*}
Using that $A$ is convex and open, we deduce by the mean value theorem
that there exists a positive constant $K$ and a compact set $B\subseteq
\R^d\times\R^d$ such that for all $j\in\Z^d$, $x \in(\rho j, \rho(j
+ 1)]$ and $u \in(0,1]^d$ we have
\[
\big|\phi(x, u) - \phi(\rho j + \rho u, u) \big| \leq K \rho\1_B(x, u).
\]
Thus, $D_\rho\leq K \rho\int_{(0, 1]^d} \int_{\R^d} \1_{B} (x, u) \,
\di x \, \di u$, which shows \eqref{stabconv}.

Now, we are ready to prove the statement \eqref{mainprop}. By~\eqref
{stabconv} and condition~\eqref{nj} we conclude that
\[
\bigl(\{n_j T\}, \{n_j \theta_z \}
\bigr)_{z \in\mathcal{A}} \stab(U, \eta _z)_{z \in\mathcal{A}} \quad
\text{as } j \to\infty.
\]
Next, consider the map $f : \R^d \times\R^{l'} \to\R^{d \times l'}$,
where $l'$ denotes the cardinality of $\mathcal A$, given by
\[
f(x, y_1, \ldots, y_{l'}) = \bigl(\{x + y_1
\}, \ldots, \{x + y_{l'}\}\bigr).
\]
This map is discontinuous exactly in those points $x, y_1, \ldots,
y_{l'}$ for which $x_j + y_i \in\Z$ for some $i \in\{1, \ldots, l'\}$
and some $j \in\{1, \ldots, d\}$. Note that the probability of the
limiting variable $(U, \eta_z )_{z \in\mathcal{A}}$ lying in the
latter set is $0$. Hence, it follows from the continuous mapping
theorem for stable convergence that
\[
f\bigl(\{n_j T\}, \bigl(\{n_j \theta_z\}
\bigr)_{z \in\mathcal{A}}\bigr) \stab f\bigl(U,(\eta_z)_{z \in\mathcal{A}}
\bigr) = \bigl(\{U + \eta_z\}\bigr)_{z \in
\mathcal{A}}
\]
as $j \to\infty$. Since $x = \{x\} + \lfloor x \rfloor$ we have the
identity $\{x + y\} = \{\{x\} + \{y\}\}$ and the left-hand side becomes
\[
f\bigl(\{n_j T\}, \bigl(\{n_j \theta_z\}
\bigr)_{z \in\mathcal{A}} \bigr) = \bigl(\{n_jT + n_j
\theta_z \}\bigr)_{z \in\mathcal{A}},
\]
which concludes the proof of Proposition~\ref{prop1}.
\end{proof}
Now, we introduce the notation
\begin{equation}
\label{gin} g_{i,n}(x) = \sum_{j = 0}^{k}
(-1)^j \binom{k}{j} g\bigl((i - j)/n - x\bigr),
\end{equation}
and observe the identity
\[
\Delta_{i,k}^n X= \int_{\R}
g_{i, n}(s) \, \di L_s.
\]
The next lemma presents some estimates for the function $g_{i, n}$. Its
proof is a straightforward consequence of Assumption~(B) and the Taylor
expansion.

\begin{lemma}
\label{lem1}
Suppose that Assumption~(B) holds and let $z=1,\ldots, l$. Then there
exists an $N \in\N$ such that for all $n\geq N$ and $i \in\{k, \ldots
, n\}$ the following hold:
\begin{enumerate}
\item\label{it:lem1:1} $|g_{i,n}(x)| \leq C (|i/n - x - \theta
_z|^{\alpha_z} + n^{-\alpha_z} )$ for all $x\in[\frac{i-2k}{n} - \theta
_z, \frac{i + 2k}{n} - \theta_z]$.

\item\label{it:lem1:2} $|g_{i,n}(x)| \leq Cn^{-k}|(i-k)/n-x - \theta
_z|^{\alpha_z-k}$ for all $x \in(\frac{i}{n}-\delta- \theta_z,\tfrac
{i-k}{n} - \theta_z)$ if $\alpha_z - k < 0$.

\item\label{it:lem1:3} $|g_{i,n}(x)| \leq C n^{-k} |(i-k)/n - x - \theta
_z|^{\alpha_z - k}$ for all $x \in(\frac{i + k}{n} - \theta_z, \frac{i
- k}{n} + \delta- \theta_z)$ if $\alpha_z - k < 0$.

\item\label{it:lem1:4} $|h_{k,z}(x)| \leq|x - k|^{\alpha- k}$ for all
$x\geq k+1$ and $|h_{k,z}(x)| \leq|x + k|^{\alpha- k}$ for all $x
\leq-k - 1$, if $\alpha_z - k < 0$.

\item\label{it:lem1:5} For each $\varepsilon>0$ it holds that
\begin{align*}
n^k|g_{i,n}(s)| \1_{(-\infty, \frac{i}{n} - \varepsilon- \theta_l]}(s) &\leq
C_\varepsilon \bigl(\1_{[-\theta_l - \delta', 1 - \theta_l]}(s)
\\
&\quad+ \1_{(-\infty, -\theta_l - \delta')}(s) |g^{(k)}(-s)| \bigr).
\end{align*}
\end{enumerate}
Furthermore, similar estimates hold for $z = 0$ with obvious
adjustments that account for the fact that $g$ and $h_{k,0}$ are both
vanishing on $(- \infty, 0)$.
\end{lemma}

\subsection{Proof of Theorem~\ref{th2} in the compound Poisson case}
\label{sec3.1}
In this subsection we assume that $L$ is a compound Poisson process.
Recall that $(T_m)_{m \geq1}$ denotes the jump times of $L$. Let
$\varepsilon> 0$ and consider $n_j \in\N$ such that $\varepsilon n_j
> 4k$. Define the set\querymark{Q3}
\begin{align*}
\varOmega_\varepsilon= \bigl\{&\omega\in\varOmega: \text{for all $m \in\N $
with $T_m(\omega) \in[-\theta_l, 1]$ it holds} 
\\
&|T_{m}(\omega) - T_i(\omega)| > 2\varepsilon,
T_m(\omega) + \theta_{z} - \theta_{z'} \notin
\bigl[T_i(\omega) - 2 \varepsilon, T_i(\omega) + 2
\varepsilon\bigr]
\\
&\forall i \neq m \ \forall z, z' \in\{0, \ldots, l\} \text{ and $
\Delta L_s(\omega) = 0$}
\\
&\text{for all $s \in[- \varepsilon- \theta_z, -
\theta_z + \varepsilon ] \cup[1 - \varepsilon- \theta_z,
1 -\theta_z + \varepsilon]$} \ \forall z\in\{0, \ldots, l\} \bigr\}.
\end{align*}
Roughly speaking, on the set $\varOmega_\varepsilon$ the jump times in
$[-\theta_l, 1]$ are well separated, their increments are outside a
small neighbourhood of $\theta_{z} - \theta_{z'}$, and there are no
jumps around the fixed points $-\theta_z$ and $1 - \theta_z$. In
particular, it obviously holds that $\mathbb{P}(\varOmega_\varepsilon) \to
1$ as $\varepsilon\to0$.

Throughout the proof we assume without loss of generality that $0 \in
\mathcal A$. Now, we introduce a decomposition, which is central for
the proof. Recalling the definition of $g_{i, n}$ at \eqref{gin}, we
observe the identity
\begin{equation}
\label{maindec} \Delta_{i, k}^n X = \sum
_{z \in\mathcal{A}} M_{i, n, \varepsilon, z} + \sum_{z \in\mathcal{A}^c}
M_{i, n, \varepsilon, z} + R_{i, n,
\varepsilon},
\end{equation}
where for $z = 1, \ldots, l$
\begin{align*}
M_{i, n, \varepsilon, 0} &= \int_{\frac{i}{n} - \varepsilon}^{\frac{i}{n}}g_{i, n}(s) \, \di L_s,
\qquad
M_{i, n, \varepsilon, z} = \int_{\frac{i}{n} - \theta_z - \varepsilon}^{\frac{i}{n} - \theta_z + \frac{\lfloor n\varepsilon\rfloor}{n}} g_{i,n}(s) \, \di L_s\\
R_{i ,n, \varepsilon} &= \int_{-\infty}^{\frac{i}{n} - \theta_l - \varepsilon}
g_{i, n}(s) \, \di L_s + \sum_{z = 1}^{l}\int_{\frac
{i}{n} - \theta_z + \frac{\lfloor n\varepsilon\rfloor}{n}}^{\frac
{i}{n} - \theta_{z - 1} - \varepsilon} g_{i, n}(s) \, \di
L_s.
\end{align*}
It turns out that the first term $\sum_{z \in\mathcal{A}} M_{i, n,
\varepsilon, z}$ is dominating, while the other two are negligible.

\subsubsection{Main terms in Theorem~\ref{th2}\ref{it:th2:1}}
\label{sec3.1.2}

In this subsection we consider the dominating term in the decomposition
\eqref{maindec}. We want to prove that, on $\varOmega_\varepsilon$, then
for $j \to\infty$
\begin{equation}
\label{toshow} n_j^{\alpha p} \sum_{i = k}^{n_j}
\Bigg| \sum_{z \in\mathcal{A}} M_{i, n_j, \varepsilon, z} \Bigg|^p \stab
\sum_{z \in\mathcal A} |c_z|^p \sum
_{m : T_m \in[-\theta_z, 1 - \theta_z]} |\Delta L_{T_m}|^p
V_m^z,
\end{equation}
where the limit has been introduced in \eqref{part1.2}. Let us fix an
index $z \in\mathcal{A}$. Then, on $\varOmega_\varepsilon$, for each jump
time $T_m \in(-\theta_z, 1 - \theta_z]$ there exists a unique random
variable $i_{m, z} \in\N$ such that
\[
T_m \in \biggr(\frac{i_{m, z} - 1}{n} - \theta_z,
\frac{i_{m, z}}{n} - \theta_z \biggl].
\]
We also observe the following implication, which follows directly from
the definition of the set $\varOmega_\varepsilon$:\querymark{Q4}
\[
\text{On } \varOmega_\varepsilon, \text{ if } M_{i,n,\varepsilon,z} \neq 0
\text{ for some } z \in\mathcal{A} \implies M_{i,n,\varepsilon,z'} = 0 \text{ for any }
z'\neq z \text{ in } \mathcal{A}.
\]
Indeed, this is the consequence of the definition of the term $M_{i, n,
\varepsilon, z}$ and the statement
\[
T_m(\omega) + \theta_z - \theta_{z'} \notin
\bigl[T_{m'}(\omega) - 2\varepsilon, T_{m'}(\omega) + 2
\varepsilon\bigr] \ \forall m'\neq m \ \forall z, z'
\in\{0, \ldots, l\},
\]
which holds on $\varOmega_\varepsilon$. Hence, we conclude that
\[
n^{\alpha p}\sum_{i = k}^{n} \Bigg|\sum
_{z \in\mathcal{A}} M_{i, n,
\varepsilon, z} \Bigg|^p =
n^{\alpha p} \sum_{z \in\mathcal{A}} \sum
_{i = k}^{n} | M_{i, n, \varepsilon, z} |^p
\]
on $\varOmega_\varepsilon$, and we obtain the representation
\begin{align}
&n^{\alpha p}\sum_{i = k}^{n}
|M_{i, n, \varepsilon, z}|^p = V_{n,
\varepsilon, z} \quad\text{with}
\nonumber
\\
\label{Vnz} & V_{n, \varepsilon, z} = n^{\alpha p} \sum
_{m : T_m \in
(-\theta_z, 1 - \theta_z]} |\Delta L_{T_m}|^p \sum
_{u = -\lfloor
n\varepsilon\rfloor}^{\lfloor n\varepsilon\rfloor+ v_m^z} |g_{i_{m,
z} + u, n}(T_m)|^p,
\end{align}
where $v_m^z$ are random variables taking values in $\{-2, -1, 0\}$
that are measurable with respect to $T_m$. If $z = 0$ then the sum
above is one-sided, i.e. from $u = 0$ to $\lfloor n\varepsilon\rfloor
$, cf. \cite[Eq.~(4.2)]{BLP17}. Next, we observe the identity
\[
\{nT_m + n \theta_z\} = nT_m + n
\theta_z - \lfloor nT_m + n \theta_z
\rfloor= nT_m + n \theta_z - (i_{m, z} -1).
\]
Due to Assumption~(B), we can write $g(x)= c_z|x - \theta_z|^{\alpha}
f(x)$ with $f(x) \to1$ as $x \to\theta_z$, for any $z\in\mathcal{A}$
(for $\theta_0 = 0$ we need to replace $|x|^{\alpha}$ by $x_+^{\alpha
}$). This allows us to decompose
\begin{align}
&n^{\alpha} g \biggl(\frac{i_{m, z} + u - r}{n} - T_m \biggr)\nonumber\\
&\quad = c_z n^{\alpha} \Big|\frac{i_{m, z} + u -r}{n} - T_m - \theta_z \Big|^{\alpha} f \biggl(\frac{i_{m, z} + u -r}{n} - T_m \biggr)\nonumber\\
&\quad = c_z | u - r + i_{m,z} - nT_m - n\theta_z |^{\alpha} f \biggl(\frac{u-r}{n} +n^{-1} (i_{m,z} - n T_m) \biggr)\nonumber\\
&\quad = c_z | u - r + 1 - \{nT_m + n\theta_z\} |^{\alpha} f \biggl(\frac{u-r}{n} +n^{-1} \bigl(n\theta_z + 1 - \{n T_m + n\theta_z\}\bigr) \biggr)\nonumber\\
&\quad = c_z | u - r + 1 - \{nT_m + n \theta_z\}|^{\alpha} f \biggl(\frac{u - r + 1 - \{n T_m + n\theta_z\}}{n} +\theta_z \biggr),\label{gdecomp}
\end{align}
for any $m \in\N$, $0 \leq r \leq k$ and $z \in\mathcal{A}$. Since
$f(x) \to1$ as $x \to\theta_z$, we find that for any $d \in\N$
\begin{align*}
& \biggl(n_j^{\alpha} g \biggl(\frac{i_{m, z} + u - r}{n_j} -
T_m \biggr) \biggr)_{|u|, m \leq d,\, 0\leq r \leq k,\, z \in\mathcal A}
\\
&\stab\bigl(c_z |u-r + 1 - \{U_m +
\eta_z\}|^{\alpha}\bigr)_{|u|, m \leq d,\, 0
\leq r \leq k,\, z \in\mathcal A},
\end{align*}
which holds due to condition~\eqref{nj}, decomposition~\eqref{gdecomp}
and Proposition~\ref{prop1} (for $\theta_0 = 0 $ we again need to
replace $|x|^{\alpha}$ by $x_+^{\alpha}$). Hence, by continuous mapping
theorem for stable convergence we deduce that
\begin{equation}
\label{gstab} \bigl(n_j^{\alpha} g_{i_{m,z} +u, n_j}(T_m)
\bigr)_{|u|, m \leq d,\, z \in
\mathcal A} \stab\bigl(c_z h_{k, z} \bigl(1 + u -
\{U_m + \eta_z\}\bigr) \bigr)_{|u|, m
\leq d,\, z \in\mathcal A}
\end{equation}
as $j \to\infty$,
which is a key result of the proof. We now define a truncated version
of $V_{n, \varepsilon, z}$ introduced in \eqref{Vnz}:
\[
V_{n,\varepsilon,z,d} := n^{\alpha p} \sum_{\substack{m \leq d: \\ T_m
\in(-\theta_z, 1 - \theta_z]}} |
\Delta L_{T_m}|^p \Biggl(\sum_{u =
-\lfloor\varepsilon d \rfloor}^{\lfloor\varepsilon d \rfloor+ v_m^z}
|g_{i_{m, z} + u, n}(T_m)|^p \Biggr).
\]
From \eqref{gstab} and properties of stable convergence we conclude that
\begin{equation}
\label{vndconv} (V_{n_j, \varepsilon, z, d} )_{z \in\mathcal{A}} \stab(V_{\varepsilon
, z, d})_{z \in\mathcal{A}}
\quad\text{as } j \to\infty,
\end{equation}
where
\[
V_{\varepsilon, z, d} = |c_z|^p \sum
_{\substack{m \leq d : \\ T_m \in
(-\theta_z, 1 - \theta_z]}} |\Delta L_{T_m}|^p \Biggl(\sum
_{u =
-\lfloor\varepsilon d \rfloor}^{\lfloor\varepsilon d \rfloor+ v_m^z} \big|h_{k,z} \bigl(1+u -
\{U_m +\eta_z\}\bigr)\big|^p \Biggr).
\]
Applying a monotone convergence argument, we deduce the almost sure convergence
\begin{equation}
\label{domconv} V_{\varepsilon, z, d} \uparrow V_z =
|c_z|^p \sum_{T_m \in(-\theta_z,
1 - \theta_z]} |\Delta
L_{T_m}|^p \biggl(\sum_{u \in\Z}
\big|h_{k,z} \bigl(1 + u - \{U_m + \eta_z\}
\bigr)\big|^p \biggr)
\end{equation}
as $d \to\infty$, where the second sum on the right-hand side is
finite, since $|h_{k, z}(x)| \leq C|x|^{\alpha- k}$ for large enough
$|x|$ and all $z \in\mathcal A$, and $\alpha< k - 1/p$. In view of
\eqref{vndconv} and \eqref{domconv}, we are left to prove the convergence
\[
\lim_{d \to\infty} \limsup_{n \to\infty} |
V_{n, \varepsilon, z, d} - V_{n, \varepsilon, z} | = 0
\]
on $\varOmega_{\varepsilon}$. Set $K_d = \sum_{m > d : T_m \in(-\theta_z,
1 - \theta_z]} |\Delta L_{T_m}|^p$ and observe that $K_d \to0$ as $d
\to\infty$, since $L$ is a compound Poisson process. Due to Lemma~\ref
{lem1} we conclude that $|n^{\alpha} g_{i, n}(x)| \leq C \min\{1, |i/n
- x|^{\alpha- k}\}$ and thus
\[
|V_{n, \varepsilon, z, d} - V_{n, \varepsilon, z} | \leq C \biggl(K_d + \sum
_{|u| > \lfloor\varepsilon d \rfloor} |u|^{p(\alpha- k)} \biggr) \quad\text{for all
} z \in\mathcal{A},
\]
and the latter converges to $0$ almost surely as $d \to\infty$,
because $\alpha< k - 1/p$. Consequently, we have shown \eqref{toshow}.
\qed

\subsubsection{Main terms in Theorem~\ref{th2}\ref{it:th2:2}}
\label{sec3.1.3}
We start with a simple lemma.

\begin{lemma}
\label{lem2}
Let $(a_i)_{i \in\N}$ be a sequence of positive real numbers such that\break
$\lim_{i \to\infty} ia_i =1$. Then it holds that
\[
\lim_{n \to\infty} \frac{1}{\log(n)} \sum
_{i = 1}^{cn} a_i =1
\]
for any fixed $c \in\N$.
\end{lemma}

\begin{proof}
Due to the assumption of the lemma, we have that $(a_i)_{i \in\N}$ is
a bounded sequence and for each $\epsilon> 0$ there exists an $N =
N(\epsilon)$ with
\[
|a_i - i^{-1}| \leq\epsilon i^{-1} \quad
\text{for all } i \geq N.
\]
It obviously holds that $\lim_{n \to\infty} \sum_{i = 1}^{cn} i^{-1} /
\log(n) = 1$. On the other hand, we obtain that
\[
\limsup_{n \to\infty} \frac{1}{\log(n)} \sum
_{i = N}^{cn} |a_i - i^{-1}|
\leq\epsilon\limsup_{n \to\infty} \frac{1}{\log(n)} \sum
_{i
= 1}^{cn} i^{-1} = \epsilon.
\]
Since $\epsilon> 0$ is arbitrary, we conclude the statement of
Lemma~\ref{lem2}.
\end{proof}
Now, we will again use the decomposition \eqref{Vnz}, which
holds on $\varOmega_{\varepsilon}$, and treat each term $V_{n, \varepsilon
, z}$ separately. We consider $z \geq1$ and we will show that
\begin{equation}
\label{apprbyh} \frac{1}{\log(n)}\sum_{u = -\lfloor n\varepsilon\rfloor}^{\lfloor
n\varepsilon\rfloor+ v_m^z}
\big|n^{\alpha}g_{i_{m, z} + u, n}(T_m) - c_z
h_{k, z}\bigl(u + 1 - \{nT_m + n \theta_z\}
\bigr) \big|^p \to0
\end{equation}
as $n \to\infty$, for any $m \in\N$. Let us first consider the case
$|u| \geq k$. Recall that we have assumed that $f_z(x) = g(x) / |x -
\theta_z|^{\alpha}$ is in $C^k((\theta_z - \delta, \theta_z + \delta))$
for any $\delta< \max_{1 \leq j \leq l}(\theta_j - \theta_{j- 1})$.
Now, due to identity \eqref{gdecomp} and Taylor expansion of order $k$,
we obtain the bound (cf. \cite[Eqs.~(4.8) and~(4.9)]{BP17})
\[
\sum_{u = -\lfloor n\varepsilon\rfloor}^{\lfloor n\varepsilon\rfloor
+ v_m^z} \big|n^{\alpha}g_{i_{m,z} + u,n}(T_m)
- c_z h_{k,z}\bigl(u+1-\{ nT_m + n
\theta_z\}\bigr) \big|^p \1_{\{|u| \geq k \}} \leq C,
\]
for any $\varepsilon< \max_{1 \leq j \leq l}(\theta_j - \theta_{j -
1})$. Since $|n^{\alpha} g_{i_{m, z} + u, n}(T_m)|$ is bounded for any
$|u| < k$ due to Lemma~\ref{lem1}, we deduce the convergence in \eqref{apprbyh}.

Next, for large enough $|u|$ we observe the bounds
\[
|q_{k, \alpha}|^p a_u \leq\big|h_{k, z}
\bigl(u + 1 - \{nT_m + n \theta_z\}\bigr)\big|^p
\leq|q_{k, \alpha}|^p a_{u - k - 1} \quad\text{where }
a_u = |u|^{-1}.
\]
Hence, by Lemma~\ref{lem2}, we conclude the convergence
\begin{equation}
\label{convforh} \frac{1}{\log(n)}\sum_{u = -\lfloor n\varepsilon\rfloor}^{\lfloor
n\varepsilon\rfloor+ v_m^z}
\big| h_{k, z}\bigl(u + 1 - \{nT_m + n \theta_z\}
\bigr) \big|^p \to2|q_{k, \alpha}|^p \quad\text{as } n\to
\infty.
\end{equation}
The same statement holds for $z=0$, but the limit becomes $|q_{k,\alpha
}|^p$, since in this setting the sum is one-sided.
We set $\| x\|_p^{p} = \sum_{i=1}^m |x_i|^p$ for any $x \in\R^m$ and
$p>0$, and recall that $\| x\|_p$ is a norm for $p \geq1$. It holds that
\begin{equation}
\label{ineq} %
\begin{aligned} |\| x\|_p^p -
\| y\|_p^p| &\leq\| x - y \|_p^p
\quad\text{ when } p \in(0, 1],
\\
|\| x\|_p - \| y\|_p| &\leq\| x - y\|_p
\quad\text{ when } p > 1. \end{aligned} %
\end{equation}
By \eqref{apprbyh}, \eqref{convforh} and~\eqref{ineq}, and taking into
account the definition of $V_{n, \varepsilon, z}$ at \eqref{Vnz}, we
readily deduce the convergence
\[
\frac{V_{n, \varepsilon,z}}{\log(n)} \toop|q_{k, \alpha} c_z|^p (1 +
\1 _{\{z\geq1\}}) \sum_{m : T_m \in[-\theta_z, 1 - \theta_z]} |\Delta
L_{T_m}|^p
\]
as $n \to\infty$, and hence
\[
n^{\alpha p}\sum_{i = k}^{n} \Bigg|\sum
_{z = 0}^{l} M_{i, n,
\varepsilon, z}
\Bigg|^p \toop|q_{k,\alpha} |^p \sum
_{z = 0}^l |c_z|^p (1 +
\1_{\{z \geq1\}} ) \sum_{m : T_m \in[-\theta_z, 1- \theta
_z]} |\Delta
L_{T_m}|^p
\]
as $n \to\infty$, on $\varOmega_{\varepsilon}$.\qed

\subsubsection{Negligible terms}
\label{sec3.1.1}
Due to inequalities at \eqref{ineq}, it suffices to show that on $\varOmega
_{\varepsilon}$
\begin{equation}
\label{neglconv} a_n \sum_{i = k}^{n}
|R_{i, n, \varepsilon}|^p \toop0 \quad\text {and} \quad a_n
\sum_{i = k}^{n} |M_{i, n, \varepsilon,z}|^p
\toop0 \quad\text{for } z \in \mathcal{A}^c,
\end{equation}
as $n \to\infty$, where $a_n = n^{\alpha p}$ in Theorem~\ref{th2}\ref
{it:th2:1} and $a_n = n^{\alpha p}/ \log(n)$ in Theorem~\ref{th2}\ref
{it:th2:2}, 
and this will prove that these terms do not affect the limits in
Theorem~\ref{th2}. At this stage we notice that outside the singularity
points the kernel function $g$ satisfies the same properties under
Assumption~(B) (resp. Assumption~(B-log)) as under Assumption~(A)
(resp. Assumption~(A-log)). Consequently, we can apply the estimates
for the term $R_{i, n, \varepsilon}$ derived in \cite[Eqs.~(4.8)
and~(4.12)]{BLP17} and \cite[Section 4]{BP17} under conditions~(A) and~(A-log)
\begin{align*}
\sup_{n \in\N, i = k, \ldots, n} n^k |R_{i, n, \varepsilon}| &< \infty
\text{ almost surely if } w \in(0, 1],
\\
\sup_{n \in\N, i = k, \ldots, n} \frac{n^k |R_{i, n, \varepsilon
}|}{(\log(n))^q} &< \infty\text{ almost
surely if } w \in(1,2],
\end{align*}
where $q$ is determined via $1 / q + 1 / w = 1$, since $R_{i, n,
\varepsilon}$ is only affected by the function $g$ outside the
singularity points $\theta_z$. We readily conclude the first
convergence at \eqref{neglconv} in the setting of Theorem~\ref{th2}\ref
{it:th2:1}, because $\alpha<k-1/p$. It also holds in the setting of
Theorem~\ref{th2}\ref{it:th2:2}, where for $w \in(1,2]$ we use the
assumption that $1 / p + 1 / w > 1$.

Now, we show the second statement of \eqref{neglconv}, which is only
relevant in the setting of Theorem~\ref{th2}\ref{it:th2:1}. Since
$\alpha_z < k -1/p$ for all $z$, we can apply to $\sum_{i = k}^{n}
|M_{i, n, \varepsilon, z}|^p$, $z \in\mathcal{A}^c$, the same
techniques as for $\sum_{i = k}^{n} |M_{i, n, \varepsilon, z}|^p$, $z
\in\mathcal{A}$. Hence, using the same methods as in Section~\ref
{sec3.1.2}, we conclude that on $\varOmega_{\varepsilon}$
\[
n^{\alpha p} \sum_{i = k}^{n}
|M_{i, n, \varepsilon, z}|^p = O_{\mathbb
{P}} \bigl(n^{p(\alpha- \alpha_z)}
\bigr) \quad\text{for all } z \in \mathcal{A}^c,
\]
where the notation $Y_n = O_{\mathbb{P}} (a_n)$ means that the sequence
$a_n^{-1} Y_n$ is tight. Since $\alpha_z > \alpha$ for all $z \in
\mathcal{A}^c$, we obtain the second statement of \eqref{neglconv}. The
results of Sections~\ref{sec3.1.2}--\ref{sec3.1.1} and the fact that
$\varOmega_{\varepsilon} \uparrow\varOmega$ as $\varepsilon\to0$ imply the
assertion of Theorem~\ref{th2} in the compound Poisson case.\qed

\subsection{Proof of Theorem~\ref{th2} in the general case}
\label{sec3.2}
Let now $(L_t)_{t\in\R}$ be a general symmetric pure jump L\'evy
process with Blumenthal--Getoor index $\beta$. We denote by $N$ the
corresponding Poisson random measure defined by $N(A) := \#\{t \in\R:
(t, \Delta L_t) \in A\}$ for all measurable $A \subseteq\R\times(\R
\setminus\{0\})$. Next, we introduce the process
\[
X_t(m) = \int_{(-\infty, t] \times[-\frac{1}{m}, \frac{1}{m}]} x \bigl(g(t - s) -
g_0(-s) \bigr) \, N(\di s, \di x),
\]
which only involves small jumps of $L$. We will prove that
\begin{equation}
\label{smalljumps} \lim_{m \to\infty} \limsup_{n \to\infty}
\mathbb{P} \bigl( a_n V\bigl(X(m), p; k\bigr)_n > \epsilon
\bigr) = 0 \quad\text{for any } \epsilon>0,
\end{equation}
where $a_n = n^{\alpha p}$ in Theorem~\ref{th2}\ref{it:th2:1} and $a_n
= n^{\alpha p}/ \log(n)$ in Theorem~\ref{th2}\ref{it:th2:2}. First, due
to Markov's~inequality and the stationary increments of $X_t(m)$, it
follows that
\[
\mathbb{P}\bigl(a_n V\bigl(X(m), p; k\bigr)_n > \epsilon
\bigr) \leq\epsilon^{-1} a_n \sum
_{i = k}^{n} \E\bigl[| \Delta_{i, k}^n
X(m) |^p\bigr] \leq\epsilon^{-1} b_n \E\bigl[|
\Delta_{k, k}^n X(m) |^p\bigr],
\]
where $b_n = n a_n$. Hence it is enough to prove that
\begin{equation}
\label{eq:Ynm} \lim_{m \to\infty} \limsup_{n \to\infty} \E
\bigl[| Y_{n, m} |^p\bigr] = 0 \quad \text{where} \quad
Y_{n, m} = b_n^{1 / p} \Delta_{k, k}^n
X(m).
\end{equation}
Notice the representation
\[
Y_{n, m} = \int_{(-\infty, \frac{k}{n}] \times[-\frac{1}{m}, \frac
{1}{m}]} \bigl(b_n^{1/p}
g_{k, n}(s)\bigr) x \, N(\di s, \di x).
\]
Using this together with \cite[Theorem~3.3]{RR89}, \eqref{eq:Ynm} will
follow if
\begin{align*}
&\lim_{m \to\infty} \limsup_{n \to\infty}
\xi_{n, m} = 0 \quad\text {where} \quad\xi_{n, m} = \int
_{| x | \leq\frac{1}{m}} \chi_n(x) \, \nu(\di x) \quad\text{and}
\\
&\chi_n(x) = \int_{-\infty}^{\frac{k}{n}} \bigl( |
b_n^{1 / p} g_{k,
n}(s) x |^p
\1_{\{| b_n^{1 / p} g_{k,n}(s) x | \geq1 \}}
\\
&\quad\ \qquad+ | b_n^{1 / p} g_{k, n}(s) x
|^2 \1_{\{| b_n^{1 /
p} g_{k,n}(s) x | < 1 \}} \bigr) \, \di s.
\end{align*}
Suppose there exists a constant $K \geq0$ such that for all large $n
\in\N$
\begin{equation}
\label{eq:chi_leq} \chi_n(x) \leq K \bigl(| x |^p + | x
|^2\bigr) \quad\text{for all $x \in[-1, 1]$},
\end{equation}
then the dominated convergence theorem implies that
\[
\limsup_{m \to\infty} \Bigl[\limsup_{n \to\infty}
\xi_{n, m} \Bigr] \leq K \limsup_{m \to\infty} \int
_{| x | \leq\frac{1}{m}} \bigl(| x |^p + | x |^2\bigr)
\, \nu(\di x) = 0,
\]
using the assumption that $p > \beta$. We consider only \eqref
{eq:chi_leq} in the case of Theorem~\ref{th2}\ref{it:th2:1} as~\ref
{it:th2:2} is very similar, see \cite{BP17}. In the case of~\ref
{it:th2:1} then $\smash{b_n^{1/p}} = n^{\alpha+ 1 / p}$. For short
notation define $\varPhi_p : \R\to\R_+$ as the function
\[
\varPhi_p(y) = | y |^2 \1_{\{| y | \leq1\}} + | y
|^p \1_{\{| y | > 1\}}, \quad y \in\R.
\]
Note that $\varPhi_p$ is of modular growth, i.e. there exists a constant
$K_p > 0$ depending only on $p$ such that 
$\varPhi_p(x + y) \leq K_p(\varPhi_p(x) + \varPhi_p(y))$
for any $x,y \in\R$. We consider the
following decomposition
\begin{align*}
\chi_n(x) &= \int_{\frac{k}{n} - \frac{1}{n}}^{\frac{k}{n}} \varPhi
_p\bigl(n^{\alpha+ 1/p} g_{k, n}(s) x\bigr) \, \di s + \sum
_{z = 1}^{l} \int_{\frac{k}{n} - \theta_z - \frac{1}{n}}^{\frac{k}{n} - \theta_z + \frac
{1}{n}}
\varPhi_p\bigl(n^{\alpha+ 1/p} g_{k, n}(s) x\bigr) \, \di s
\\
&\quad+ \sum_{z = 1}^{l} \int
_{\frac{k}{n} - \theta_z + \frac
{1}{n}}^{\frac{k}{n} - \theta_{z-1} - \frac{1}{n}} \varPhi_p
\bigl(n^{\alpha+
1/p} g_{k, n}(s) x\bigr) \, \di s
\\
&\quad+ \int_{\frac{k}{n} - \theta_l - \delta}^{\frac{k}{n} - \theta_l
- \frac{1}{n}} \varPhi_p
\bigl(n^{\alpha+ 1/p} g_{k, n}(s)x\bigr) \, \di s
\\
&\quad+ \int_{-\infty}^{\frac{k}{n} - \theta_l - \delta} \varPhi _p
\bigl(n^{\alpha+ 1/p} g_{k, n}(s) x\bigr) \, \di s
\\
&=: I_0(x) + \sum_{z = 1}^{l}
I_{1, z}(x) + \sum_{z = 1}^{l}
I_{2,
z}(x) + I_3(x) + I_4(x).
\end{align*}
We treat the five types of terms separately.

\paragraph{Estimation of $I_0$} By Lemma~\ref{lem1}
\[
| g_{k, n}(x) | \leq K \bigl(| \tfrac{k}{n} - s |^{\alpha_0}
\bigr) \quad\text {for all $s \in\bigl[\tfrac{k}{n} - \tfrac{1}{n},
\tfrac{k}{n}\bigr]$}.
\]
Since $\varPhi_p$ is increasing on $\R_+$ and $\alpha\leq\alpha_0$ it
follows that
\[
I_0(x) \leq K \int_{0}^{\frac{1}{n}}
\varPhi_p\bigl(x n^{\alpha+ 1/p} s^{\alpha
_0} \bigr) \, \di s
\leq K \int_{0}^{\frac{1}{n}} \varPhi_p\bigl(x
n^{\alpha+ 1/p} s^{\alpha} \bigr) \, \di s.
\]
By elementary integration it follows that
\begin{align*}
&\int_{0}^{\frac{1}{n}} | x n^{\alpha+ 1/p}
s^{\alpha} |^2 \1_{\{ | x
n^{\alpha+ 1/p} s^\alpha| \leq1\}} \, \di s
\\
&\quad\leq K \bigl(x^2 \1_{\{ | x | \leq n^{-1/p} \}} n^{2/p - 1} +
\1_{\{
| x | > n^{-1 / p} \}} | x |^{-1/\alpha} n^{-1 - 1/(\alpha p)}\bigr)
\\
&\quad\leq K\bigl(x^2 + | x |^p\bigr).
\end{align*}
The second term in $\varPhi_p$ is dealt with as follows:
\[
\int_{0}^{\frac{1}{n}} | x n^{\alpha+ 1/p}
s^{\alpha} |^p \1_{\{| x
n^{\alpha+ 1/p} s^{\alpha} | > 1\}} \, \di s \leq| x
|^p n^{\alpha p
+ 1} \int_{0}^{\frac{1}{n}}
s^{\alpha p} \, \di s = \frac{| x
|^p}{\alpha p + 1}.
\]
Combining the two estimates above it follows that $I_0(x) \leq K(|x|^2
+ |x|^p)$.

\paragraph{Estimation of $I_{1, z}$} Similarly as for $I_0$, we have,
using arguments as in part~\ref{it:lem1:1} of Lemma~\ref{lem1}, that
\[
| g_{k,n}(s) | \leq K \sum_{j = 0}^{k}
| \tfrac{k - j}{n} - s - \theta _z |^{\alpha_z} \quad\text{for
all $s \in\bigl[\tfrac{k}{n} - \theta_z - \tfrac{1}{n},
\tfrac{k}{n} - \theta_z + \tfrac{1}{n}\bigr]$}.
\]
Using the modular growth of $\varPhi_p$ it follows that
\begin{align*}
&\int_{\frac{k}{n} - \theta_z - \frac{1}{n}}^{\frac{k}{n} - \theta_z +
\frac{1}{n}} \varPhi_p
\bigl(n^{\alpha+ 1/p} g_{k, n}(s) x\bigr) \, d s
\\
&\quad\leq K_p \sum_{j = 0}^{k}
\int_{\frac{k}{n} - \theta_z - \frac
{1}{n}}^{\frac{k}{n} - \theta_z + \frac{1}{n}} \varPhi_p
\bigl(n^{\alpha+ 1/p} |\tfrac{k - j}{n} - s - \theta_z|^{\alpha_z}
x\bigr) \, \di s
\\
&\quad= K_p \sum_{j = 0}^{k}
\int_{-\frac{j}{n} - \frac{1}{n}}^{-\frac
{j}{n} + \frac{1}{n}} \varPhi_p
\bigl(n^{\alpha+ 1/p} |s|^{\alpha_z} x\bigr) \, \di s
\\
&\quad\leq K_p \int_{-\frac{k + 1}{n}}^{\frac{k + 1}{n}} \varPhi
_p\bigl(n^{\alpha+ 1/p} |s|^{\alpha} x\bigr) \, \di s
\\
&\quad= K_p \int_{0}^{\frac{k + 1}{n}}
\varPhi_p\bigl(n^{\alpha+ 1/p} |s|^{\alpha} x\bigr) \, \di s.
\end{align*}
As for $I_0$, we get $I_{1, z}(x) \leq K(|x|^2 + |x|^p)$.

\paragraph{Estimation of $I_{2, z}$} We decompose $I_{2, z}$ into three
terms corresponding to whether we are close to the singularity $\theta
_z$ from the right or close to the singularity $\theta_{z-1}$ from the
left or in between them, but bounded away from both. More specifically,
we decompose as
\begin{align*}
I_{2,z}(x) &= \int_{\frac{k}{n} - \theta_z + \frac{1}{n}}^{\frac{k}{n}
- \theta_z + \delta}
\varPhi_p\bigl(n^{\alpha+ 1/p} g_{k, n}(s) x\bigr) \, \di s
\\
&\quad+ \int_{\frac{k}{n} - \theta_z + \delta}^{\frac{k}{n} - \theta_{z
- 1} - \delta} \varPhi_p
\bigl(n^{\alpha+ 1/p} g_{k, n}(s) x\bigr) \, \di s
\\
&\quad+ \int_{\frac{k}{n} - \theta_{z-1} - \delta}^{\frac{k}{n} - \theta
_{z - 1} - \frac{1}{n}} \varPhi_p
\bigl(n^{\alpha+ 1/p} g_{k, n}(s) x\bigr) \, \di s =:
I_{2,z}^{l}(x) + I_{2,z}^{b}(x) +
I_{2,z}^{r}(x).
\end{align*}
First we note that arguments similar to Lemma~\ref{lem1}\ref{it:lem1:3}
imply that
\[
| g_{k, n}(s) | \leq K n^{-k} |\tfrac{k}{n} - s -
\theta_z|^{\alpha_z -
k} \quad\text{for all $s \in\bigl[
\tfrac{k}{n} - \theta_z + \tfrac{1}{n}, \tfrac{k}{n} -
\theta_z + \delta\bigr]$.}
\]
Using again that $\varPhi_p$ is decreasing on $\R_+$ it follows that
\begin{align*}
I_{2, z}^{l}(x) &\leq K \int_{\frac{k}{n} - \theta_z + \frac
{1}{n}}^{\frac{k}{n} - \theta_z + \delta}
\varPhi_p\bigl(n^{\alpha+ 1/p - k} | \tfrac{k}{n} - s -
\theta_z |^{\alpha_z - k} x\bigr) \, \di s
\\
&\leq K \int_{\frac{1}{n}}^{\delta} \varPhi_p
\bigl(n^{\alpha+ 1/p - k} |s|^{\alpha_z - k} x\bigr) \, \di s.
\end{align*}
If $\alpha_z = k - 1/2$ then
\begin{align*}
\int_{\frac{1}{n}}^{\delta} | x n^{\alpha+ 1/p - k}
s^{\alpha_z -
k}|^2 \1_{\{ | x^2 n^{\alpha+ 1/p - k} s^{\alpha_z - k} | \leq1 \}} \, \di s &\leq
x^2 n^{2(\alpha+ 1/p - k)} \int_{\frac{1}{n}}^{\delta}
s^{-1} \, \di s
\\
&\leq K x^2,
\end{align*}
where we used that $\alpha< k - 1/p$. For $\alpha_z \neq k - 1 / 2$ we
have that
\begin{align*}
&\int_{\frac{1}{n}}^{\delta} | x n^{\alpha+ 1/p - k}s^{\alpha_z -k}|^2 \1_{\{| x n^{\alpha+ 1/p - k} s^{\alpha_z - k}| \leq1\}} \, \di s\\
&\quad \leq K\bigl(|x|^2 n^{2(\alpha+ 1/p - k)} + |x|^2 n^{2(\alpha- \alpha_z) + 2/p - 1} \1_{\{| x | \leq n^{-1/p} \}}\\
&\qquad+ |x|^{\frac{1}{k - \alpha_z}} n^{\frac{\alpha+ 1/p - k}{k - \alpha_z}} \1_{\{|x| > n^{-1/p}\}}\bigr)\\
&\quad \leq K\bigl(x^2 + |x|^p\bigr),
\end{align*}
where we used that $\alpha\leq\alpha_z < k - 1/p$. Moreover,
\[
\int_{\frac{1}{n}}^{\delta} |x n^{\alpha+ 1/p - k}
s^{\alpha_z - k}|^p \1_{\{| x n^{\alpha+ 1/p - k} s^{\alpha_z - k}| > 1\}} \, \di s \leq K
|x|^p.
\]
The term $I_{2,z}^{r}$ is handled similarly. For the last term $I_{2,
z}^b$ we note that, since we are bounded away from both $\theta_{z-1}$
and $\theta_z$, there exists a constant $K > 0$ such that
\begin{equation*}
|g_{k,n}(s)| \leq K n^{-k} \quad\text{for all $s \in\bigl[
\tfrac{k}{n} - \theta_z + \delta, \tfrac{k}{n} -
\theta_{z-1} - \delta\bigr]$.}
\end{equation*}
This readily implies the bound $I_{2,z}^{b}(x) \leq K(x^2 + |x|^p)$.

\paragraph{Estimation of $I_3$} Arguments as in Lemma~\ref{lem1} imply that
\begin{equation*}
|g_{k, n}(s)| \leq K n^{-k}|\tfrac{k}{n} - s -
\theta_z|^{\alpha_l - k} \quad\text{for all $s \in\bigl[
\tfrac{k}{n} - \theta_l - \delta, \tfrac
{k}{n} -
\theta_l - \tfrac{1}{n}\bigr]$.}
\end{equation*}
One may then proceed as for the term $I_{2,z}^l$ above to conclude that
$I_3(x) \leq K(x^2 + |x|^p)$.

\paragraph{Estimation of $I_4$:} First we decompose the integral\querymark{Q5} into two sub-integrals:
\begin{align*}
\int_{-\infty}^{\frac{k}{n} - \theta_l - \delta} \varPhi_p
\bigl(n^{\alpha+
1/p} g_{k, n}(s) x\bigr) \, \di s &= \int
_{-\delta' - \theta_l}^{\frac{k}{n}
- \delta- \theta_l} \varPhi_p
\bigl(n^{\alpha+ 1/p} g_{k, n}(s)x \bigr) \, \di s
\\
&\quad+ \int_{-\infty}^{-\delta' - \theta_l} \varPhi_p
\bigl(n^{\alpha+ 1/p} g_{k, n}(s) x\bigr) \, \di s.
\end{align*}
In the first integral we are bounded away from $\theta_l$, hence
$|g_{k,n}(s)| \leq K n^{-k}$ for all $s$ in the interval $[-\delta' -
\theta_l, \frac{k}{n} - \delta- \theta_l]$.
For the latter integral note first that by Lemma~\ref{lem1}\ref{it:lem1:5}
\begin{equation*}
\int_{-\infty}^{-\delta' - \theta_l} \varPhi_p
\bigl(n^{\alpha+ 1/p} g_{k,
n}(s) x\bigr) \, \di s \leq\int
_{-\infty}^{-\delta' - \theta_l} \varPhi _p
\bigl(n^{\alpha+ 1/p - k} |g^{(k)}(-s)| x\bigr) \, \di s.
\end{equation*}
Now
\begin{align*}
&\int_{\delta' + \theta_l}^{\infty} | x n^{\alpha+ 1/p - k} g^{(k)}(s) |^2 \1_{\{ |x n^{\alpha+ 1/p - k} g^{(k)}(s) | \leq1 \}} \, \di s \\
&\quad \leq|x n^{\alpha+ 1/p - k}|^2 \int_{\delta' + \theta_l}^{\infty}|g^{(k)}(s)|^2 \, \di s.
\end{align*}
Since $|g^{(k)}|$ is decreasing on $(\theta_l + \delta', \infty)$ and
$g^{(k)} \in L^w((\theta_l + \delta', \infty))$ for some $w \leq2$ it
follows that the last integral is finite. Lastly, we find for $x \in
[-1 , 1]$ that
\begin{align*}
&\int_{\theta_l + \delta'}^{\infty} |x n^{\alpha+ 1/p - k}
g^{(k)}(s)|^p \1_{\{|x n^{\alpha+ 1/p - k} g^{(k)}(s)| > 1\}} \, \di s
\\
&\quad \leq|x|^p n^{p(\alpha+ 1/p - k)} \int_{\delta' + \theta_l}^{\infty}
|g^{(k)}(s)|^p \1_{\{|g^{(k)}(s)| > 1 \}} \, \di s.
\end{align*}
By our assumptions the last integral is finite, indeed
\begin{equation*}
\int_{\delta' + \theta_l}^{\infty} |g^{(k)}(s)|^p
\1_{\{|g^{(k)}(s)| >
1\}} \, \di s \leq K_p \| g^{(k)}
\|_{L^w((\delta' + \theta, \infty
))}^w < \infty.
\end{equation*}

\subsubsection{Negligibility of small jumps}

Now, we note that $X_t - X_t(m)$ is the integral \eqref{Xmodel}, where
the integrator is a compound Poisson process that corresponds to big
jumps of $L$. Hence, we obtain the results of Theorem~\ref{th2} for the
process $X-X(m)$ as in Section~\ref{sec3.1}. More specifically, under
assumptions of Theorem~\ref{th2}\ref{it:th2:1} it holds that
\[
n_j^{\alpha p}V\bigl(X - X(m), p; k\bigr)_{n_j} \stab
\sum_{z \in\mathcal A} |c_z|^p \sum
_{r : T_r \in[-\theta_z, 1 - \theta_z]} |\Delta L_{T_r}|^p
\1_{\{|\Delta L_{T_r}| > \frac{1}{m}\}} V_r^z
\]
where $V_r^z$ has been defined at \eqref{part1.2}. The term on the
right-hand side converges to the limit of Theorem~\ref{th2}\ref
{it:th2:1} as $m \to\infty$, since
\begin{equation*}
\sum_{r : T_r \in[-\theta_z, 1 - \theta_z]} |\Delta L_{T_r}|^p
< \infty\quad\text{for any $p > \beta$.}
\end{equation*}
Finally, using the decomposition $X=(X-X(m)) + X(m)$ and letting first
$n_j \to\infty$ and then $m\to\infty$, we deduce the statement of
Theorem~\ref{th2} by \eqref{smalljumps} and the inequalities \eqref
{ineq}. This completes the proof.\qed

\begin{acknowledgement}[title={Acknowledgments}]
The authors acknowledge financial support from the project
``Ambit fields: probabilistic properties and statistical inference''
funded by Villum Fonden.
\end{acknowledgement}


%




\end{document}